\documentclass[11pt,reqno]{amsart}
\usepackage{amssymb, amsmath, amsthm, mathrsfs, setspace, enumerate}
\usepackage{booktabs,tikz}
\usetikzlibrary{arrows.meta}
\usepackage[bookmarksnumbered, colorlinks, plainpages]{hyperref}
\usepackage{array}
\usepackage{tabularx}
\usepackage{tikz}
\usepackage{mathrsfs}
\newtheorem*{theo3A}{Theorem 3.A}
\newtheorem*{theo3B}{Theorem 3.B}
\newtheorem*{theo3C}{Theorem 3.C}
\newtheorem*{theo3D}{Theorem 3.D}

\newtheorem*{lem2A}{Lemma 2.A}
\newtheorem*{lem2B}{Lemma 2.B}

\newtheorem*{propo1A}{Proposition 1.A}
\newtheorem*{propo1B}{Proposition 1.B}
\newtheorem*{propo1C}{Proposition 1.C}

\newtheorem*{defi1A}{Definition 1.A}
\newtheorem*{defi1B}{Definition 1.B}
\newtheorem*{defi1C}{Definition 1.C}

\newtheorem*{defi1D}{Definition 1.D}
\newtheorem*{defi1E}{Definition 1.E}
\newtheorem*{defi1F}{Definition 1.F}
\newtheorem*{defi1G}{Definition 1.G}
\newtheorem*{defi1H}{Definition 1.H}

\newtheorem*{exm3A}{Example 3.A}

\newtheorem{theo}{Theorem}[section]
\newtheorem{lem}{Lemma}[section]

\newtheorem{exm}{Example}[section]
\newtheorem{defi}{Definition}[section]
\newtheorem{rem}{Remark}[section]
\newtheorem{propo}{Proposition}[section]

\newcommand{\ol}{\overline}

\newcommand{\be}{\begin{equation}}
\newcommand{\ee}{\end{equation}}
\newcommand{\beas}{\begin{eqnarray*}}
\newcommand{\eeas}{\end{eqnarray*}}
\newcommand{\bea}{\begin{eqnarray}}
\newcommand{\eea}{\end{eqnarray}}

\numberwithin{equation}{section}
\begin{document}

\title[B\MakeLowercase{orel Exceptional Values in Several Complex Variables}......]{\LARGE B\Large\MakeLowercase{orel Exceptional Values in Several Complex Variables\newline and Their Applications to Shared Values of Shifts\newline \vspace {.25cc} and Difference Operators}}

\date{}
\author[A. B\MakeLowercase{anerjee}, S. M\MakeLowercase{ajumder}, J. B\MakeLowercase{anerjee}]{A\MakeLowercase{bhijit} B\MakeLowercase{anerjee}$^1$, S\MakeLowercase{ujoy} M\MakeLowercase{ajumder}$^2$ , J\MakeLowercase{hilik} B\MakeLowercase{anerjee}$^3$$^{*}$}
\address{$^{1}$ Department of Mathematics, University of Kalyani, West Bengal 741235, India.}
\email{abanerjee\_kal@yahoo.co.in}
\address{$^{2}$Department of Mathematics, Raiganj University, Raiganj, West Bengal-733134, India.}
\email{sm05math@gmail.com, sjm@raiganjuniversity.ac.in}
	\address{$^{3}$ Department of Mathematics, University of Kalyani, West Bengal 741235, India.}
\email{jhilikbanerjee38@gmail.com}

\renewcommand{\thefootnote}{}
\footnote{*\emph{Corresponding Author}: Jhilik Banerjee.}

\renewcommand{\thefootnote}{\arabic{footnote}}
\setcounter{footnote}{0}

\begin{abstract} 
We study value-sharing problems for shifts and higher-order difference operators of meromorphic functions in $\mathbb{C}^n$. Using Nevanlinna theory in $\mathbb{C}^n$, we obtain new uniqueness theorems when functions share values counting or ignoring multiplicities, extending several classical one-variable results to higher dimensions. A key contribution of this work appears in Section 2, where we establish fundamental results on Borel exceptional values in several complex variables. These propositions provide the main tools for proving our principal theorems. As applications, we derive conditions ensuring that a transcendental entire function satisfies $\Delta_c^{k} \equiv d\hspace{.05cc} f$ and we study meromorphic solutions of certain partial differential-difference equations, obtaining growth estimates and structural descriptions of entire solutions. To the best of our knowledge, this is the first systematic study of such shared value problems for higher-order difference operators in several complex variables.
\end{abstract}

\thanks{Typeset by \AmS -\LaTeX}
\maketitle

\section{{\bf Introduction}}

\noindent
The theory of Borel exceptional values occupies an important position in the value distribution theory of entire and meromorphic functions. Originating from the classical theory of meromorphic functions of one complex variable, the concept provides a quantitative description of the values that a meromorphic function assumes comparatively rarely in relation to its rate of growth. In the framework of Nevanlinna theory, the distribution of the pre-images of a complex value is closely connected with the characteristic function, the counting function, and the deficiency of the meromorphic function. In particular, a value is called a Borel exceptional value if the corresponding counting function is of smaller growth than the Nevanlinna characteristic. Such values reveal important information about the value distribution and structural properties of meromorphic functions.

\noindent
With the development of Nevanlinna theory in several complex variables, the notion of Borel exceptional values has been extended to holomorphic and meromorphic functions defined in $\mathbb C^n$, where $n\geq 2$. Unlike the one-dimensional case, the inverse image of a complex value under a nonconstant holomorphic function in several complex variables is generally an analytic hypersurface rather than a discrete set. Consequently, the study of Borel exceptional values requires the use of several-variable Nevanlinna theory, in which the characteristic and counting functions are formulated in terms of the growth of meromorphic functions and the associated analytic divisors. This extension provides a natural framework for investigating the relationship between the growth of meromorphic functions and the distribution of their level sets in higher-dimensional complex spaces.

\noindent
In recent years, the study of value distribution in several complex variables has developed in several directions, particularly in connection with uniqueness theory, normal family theory, partial differential equations, and difference operators. Borel exceptional values provide useful tools in these investigations because restrictions on the distribution of certain level sets can lead to strong conclusions concerning the growth and structural behaviour of meromorphic functions. In particular, the existence of Borel exceptional values may be exploited to establish uniqueness theorems for meromorphic mappings, investigate the value-sharing properties of entire functions and their partial derivatives, and characterize solutions of certain nonlinear partial differential equations.

\noindent
Another interesting direction concerns the interaction between Borel exceptional values and shift or difference operators. Let $f$ be a meromorphic function in $\mathbb{C}^n$. For $z=(z_1,z_2,\ldots,z_n)$ and
$c=(c_1,c_2,\ldots,c_n)\in\mathbb{C}^n$, we write $z+c=(z_1+c_1,z_2+c_2,\ldots,z_n+c_n).$
The shift of $f$ by $c$ is defined as $f(z+c)=f(z_1+c_1,z_2+c_2,\ldots,z_n+c_n)$,
and the corresponding difference operator is given by
\[
\Delta_c f(z)=f(z+c)-f(z).
\]
Inductively, for $m\ge2$, we define the higher-order difference operator by
\begin{align*}
	\Delta_c^k f(z)=\Delta_c^{k-1}f(z+c)-\Delta_c^{k-1}f(z)=\Delta_c\circ\Delta_c^{k-1}f(z),
\end{align*}
where $m\in\mathbb{N}$.

\noindent
The study of the value distribution of $f$ in relation to its shifts and higher-order differences has become an active area of research. In this setting, Borel exceptional values may provide useful information about the distribution of zeros of difference expressions and the possible growth of entire and meromorphic solutions of functional equations. Similarly, their applications to partial differential-difference equations offer opportunities to establish new uniqueness results and characterize finite-order transcendental solutions.

\noindent
We define $\mathbb{Z}_+=\mathbb{Z}[0,+\infty)=\{n\in \mathbb{Z}: 0\leq n<+\infty\}$ and $\mathbb{Z}^+=\mathbb{Z}(0,+\infty)=\{n\in \mathbb{Z}: 0<n<+\infty\}$. Let $f$ be a meromorphic function on $\mathbb{C}^n$. 
Define
\begin{align*}
	\displaystyle 
	\partial_{z_i}(f(z))=\frac{\partial f(z)}{\partial z_i},\ldots,\partial^{l_i}_{z_i}(f(z))=\frac{\partial^{l_i} f(z)}{\partial z_i^{l_i}}\;\;
	\text{and}\;\;  
	\displaystyle 
	\partial^{I}(f(z))=\frac{\partial^{|I|}f(z)}{\partial z_1^{i_1}\cdots \partial z_n^{i_n}},
\end{align*}
where $l_i\in \mathbb{Z}_+$, $i=1,2,\ldots,n$ and $I=(i_1,\ldots,i_n)\in\mathbb{Z}^n_+$ such that $|I|=\sum_{j=1}^n i_j$.
\smallskip

\noindent
Furthermore, the concept has potential applications to the study of Fermat-type equations and nonlinear partial differential equations in several complex variables. For example, equations involving expressions such as
\[ \left(\frac{\partial f}{\partial z_j}\right)^p +f^q(z+c)=1,\]
where $p,q$ are positive integers, combine partial differentiation with shift operators. Investigating the distribution of the zeros of the functions involved, together with the existence of Borel exceptional values, may yield useful restrictions on the possible solutions of such equations.
\smallskip

\noindent
Motivated by these developments, the study of Borel exceptional values in several complex variables provides an interesting meeting point of classical Nevanlinna theory, value-sharing problems, and the theory of differential and difference equations. The principal objective is to develop fundamental results concerning the existence and distribution of Borel exceptional values and to explore their applications to uniqueness problems, partial derivatives, shift operators, and nonlinear partial differential-difference equations. Such investigations are expected to contribute to a deeper understanding of the growth, value distribution, and structural properties of holomorphic and meromorphic functions in several complex variables.

\subsection*{Preliminaries on Nevanlinna Theory in Several Complex Variables}
This section collects the fundamental concepts and notation from Nevanlinna theory in several complex variables that will be used throughout the paper. Our presentation follows the framework developed by Stoll \cite{Stoll-1974}, Nishino \cite{Nishino-2001}, and Hu-Li-Yang \cite{Hu-Li-Yang-2003}.
\subsection*{Hermitian Structure and Differential Forms}

Let $W$ be an $n$-dimensional Hermitian vector space equipped with a positive definite Hermitian form $(\cdot \mid \cdot)$. The induced norm is
$\|z\| = \sqrt{(z \mid z)}$, $z \in W$.
When $W=\mathbb{C}^n$, the Hermitian product is taken to be
\[
(a \mid b)=\sideset{}{_{j=1}^n}{\sum} a_j \overline{b_j},
\quad a,b\in\mathbb{C}^n.
\]

\noindent
For any subset $A\subset W$ and $r\ge0$, define (see \cite[pp. 6]{Stoll-1974})
$A[r]=\{z\in A:\|z\|\le r\}$, $A(r)=\{z\in A:\|z\|<r\}$ and $A\langle r\rangle=\{z\in A:\|z\|=r\}$.
We introduce the functions $\tau_0(z)=\|z\|$ and $\tau_W(z)=\|z\|^2.$
On $W$, the exterior derivative decomposes as $d=\partial+\bar{\partial}$, and we set
$d^c=\frac{\iota}{4\pi}(\bar{\partial}-\partial)$ and $dd^c=\frac{\iota}{2\pi}\partial\bar{\partial}$.
The standard \text{K\"ahler} form on $W$ is $\upsilon_W = dd^c \tau_W >0$.
On $W\setminus\{0\}$, we further define
$\omega_W = dd^c \log \tau_W \ge 0$ and $\sigma_W = d^c\log\tau_W \wedge \omega_W^{n-1}$,
where $n=\dim(W)$ (see \cite[pp. 6]{Stoll-1974}).

%\subsection*{Projective Space}
%\begin{defi1A}\cite[pp. 7]{Nishino-2001}
%	Let $z',z''\in\mathbb{C}^{n+1}\setminus\{0\}$. We say that $z'$ and $z''$ are equivalent if there exists $c\in\mathbb{C}\setminus\{0\}$ such that
%	\[
%	z_j''=c z_j', \quad j=0,1,\ldots,n.
%	\]
%	The equivalence classes in $\mathbb{C}^{n+1} \backslash  \{0\}$ form an 
%$n$-dimensional space called the complex projective space, which we denote by
%	\[
%	\mathbb{P}^n=\mathbb{P}(\mathbb{C}^{n+1}).
%	\]
%	Identify the compactified plane $\mathbb{C}\cup\{\infty\}$ with $\mathbb{P}^1=\mathbb{P}(\mathbb{C}^2)$ by setting (see \cite[pp. 634-635]{Stoll-1983})
%	\[
%	\mathbb{P}(z,w)=\frac{z}{w}, \quad
%	\mathbb{P}(z,1)=z, \quad
%	\mathbb{P}(1,0)=\infty
%	\]
%	if $z\in\mathbb{C}$ and $w\in\mathbb{C}\backslash \{0\}$.
%\end{defi1A}

\subsection*{Divisors and Multiplicities}

\begin{defi1A}\cite[pp. 12]{Stoll-1974}
	Let $G\neq \varnothing$ be an open subset of $W$. Let $f$ be a holomorphic function in $W$.
Take $a\in G$. Let $G_a$ be the connectivity component of $G$ containing $a$. Assume $f\mid_{G_a}\not\equiv 0$. Then $f$ admits a local expansion
	\[
	f(z)=\sideset{}{_{\lambda=p}^{\infty}}{\sum}P_\lambda(z-a),
	\]
	where $P_\lambda$ is homogeneous of degree $\lambda$ and $P_p\not\equiv0$. The polynomials $P_{\lambda}$ depend on $f$ and $a$ only.
	The integer $\mu_f^0(a)=p$ is called the zero multiplicity of $f$ at $a$.
\end{defi1A}

\begin{defi1B}\cite[pp. 12]{Stoll-1974}
Let $f$ be a meromorphic function on $G$, where $G\neq \varnothing$ is an open subset of $W$.
Take $a\in G$ and $c\in\mathbb{P}^1=\mathbb{C}\cup\{\infty\}$. Let $G_a$ be the component of $G$ containing $a$. If $0\equiv f\mid_{G_a}\not\equiv c$, define $\mu^c_f(a)=0$. Assume $0\not\equiv f\mid_{G_a}\not\equiv c$. Then an open connected neighborhood $U$ of $a$ in $G$ and holomorphic functions $g\not\equiv 0$ and $h\not\equiv 0$ exist on $U$ such that $h. f\mid_U=g$ and $\dim g^{-1}(0)\cap h^{-1}(0)\leq n-2$, where $n=\dim(W)$. Therefore the $c$-multiplicity of $f$ is just $\mu^c_f=\mu^0_{g-ch}$ if $c\in\mathbb{C}$ and $\mu^c_f=\mu^0_h$ if $c=\infty$. The function $\mu^c_f:G\to \mathbb{Z}$ is nonnegative and is called the $c$-divisor of $f$. 
\end{defi1B}

\noindent
If $f\not\equiv0$ on each component of $G$, the divisor of $f$ is $\mu_f=\mu_f^0-\mu_f^\infty$.
The function $f$ is holomorphic if and only if $\mu_f\ge0$.

\begin{defi1C}\cite[pp. 13]{Stoll-1974}
A function $\nu:G\to \mathbb{Z}$ is said to be a divisor if and only if for each $a\in G$ an open, connected neighborhood $U$ of $a$ in $G$ and a meromorphic function $f\not\equiv 0$ exist such that $\nu\mid_{U}=\mu_f$. A divisor $\nu:G\to \mathbb{Z}$ is non-negative if and only if for each $a\in G$ an open, connected neighborhood $U$ of $a$ in $G$ and a holomorphic function $f\not\equiv 0$ exist such that $\nu\mid_{U}=\mu_f$. We denote $\operatorname{supp}\nu=\overline{\{z\in G:\nu(z)\ne0\}}$.
\end{defi1C}

\noindent
 Throughout the paper, we define $\mathbb{R}_+ = \mathbb{R}[0,\infty)$ and $\mathbb{R}^+ = \mathbb{R}(0,\infty)$.

\begin{defi1D}\cite[pp. 26]{Stoll-1974} Let $\nu\geq 0$ be a non-negative divisor on the hermitian vector space $W$. Define $A=\displaystyle \operatorname{supp}\nu$. Assume $0\not\in A$. A non-negative, increasing, integral valued function $q:\mathbb{R}^+\to \mathbb{Z}^+$ is said to be a weight for $\nu$ if and only if the integral
	\begin{align*}
		\int_{A} \nu\;\left(r/\tau_0\right)^{q\circ\tau_0+1}\omega_W^{n-1}<\infty
	\end{align*}
	converges for all $r>0$, where $n=\dim(W)$. For every given non-negative divisor $\nu$ there exists a weight function for $\nu$ (see \cite{Stoll-1953}).
\end{defi1D}

\subsection*{Counting and Characteristic Functions}

\begin{defi1E} Take $0<R\leq +\infty$. Let $\nu$ be a divisor on $W(R)$ with $A=\displaystyle \operatorname{supp}\nu$. For $t>0$, the counting function $n_{\nu}$ is defined by
\begin{align*}
 \displaystyle n_{\nu}(t)=t^{-2(n-1)}\int_{A[t]}\nu\; \upsilon_W^{n-1},
 \end{align*}
where $n=\dim(W)$.
Obviously $n_{\nu}(0)=\lim\limits_{t\to 0}n_{\nu}(t)\in\mathbb{Z}$
exists, where the Lelong number $n_{\nu}(0)$ is an integer (see \cite[Theorem 3.1]{Stoll-1974}). We know that
\begin{align*}
 \displaystyle n_{\nu}(t)=\int_{A(t)}\nu\;\omega_W^{n-1}+n_{\nu}(0).
 \end{align*}
If $\nu$ is non-negative, then $n_{\nu}$ increases. For $0<s<r<R$, we define the valence function of $\nu$ by
\begin{align*}
N_{\nu}(r)=N_{\nu}(r,s)=\int_{s}^r n_{\nu}(t)\frac{dt}{t}.
\end{align*}
\end{defi1E}

\noindent
If $0\not\in \operatorname{supp}\nu$, then $s=0$ is permitted. For $\nu=\mu_f^a$, we write $n(t,a;f)$ and $N(r,a;f)$, with the usual conventions for $a=\infty$.  

For $k\in\mathbb{N}$, define the truncated multiplicity functions on $W$ by $\mu_{f,k}^a(z)=\min\{\mu_f^a(z),k\}$.
We write the truncated counting functions $n_{\nu}(t)=n_k(t,a;f)$ if $\nu=\mu_{f,k}^a$ and $n_{\nu}(t)=\ol n(t,a;f)$ if $\nu=\mu_{f,1}^a$. 
Also the truncated valence functions $N_{\nu}(t)=N_k(t,a;f)$ if  $\nu=\mu_{f,k}^a$ and $N_{\nu}(t)=\ol N(t,a;f)$ if $\nu=\mu_{f,1}^a$.
If $a=\infty$, we write $n_k(t,a;f)=n_k(t,f)$ and $N_k(t,a;f)=N_k(t,f)$
and so on.

\noindent
A non-negative divisor $\nu:W\to \mathbb{Z}_+$ is said to be algebraic if and only if $\nu$ is  the zero divisor of a polynomial. Thus a divisor $\nu:W\to \mathbb{Z}_+$ is algebraic if and only if $n_{\nu}$ is bounded, which implies that $N_{\nu}=O(\log r)$ (see \cite[pp. 19]{Stoll-1974}).

%\subsection*{Nevanlinna Characteristic}

\begin{defi1F}\cite[pp. 16-17]{Stoll-1974}
Let $G\neq \varnothing$ be an open subset in $W$. Let $f$ be a meromorphic function in $G$ in the sense that $f$ can be written as a quotient of two relatively prime holomorphic functions. We will write $f = (f_0, f_1)$ where $f_0 \not\equiv 0$. The standard definition of the Nevanlinna characteristic function of $f$ is given by
\begin{align*}
T_f(r,s) := \int_s^r \frac{A_f(t)}{t}\, dt,
\end{align*}
where $0 < s < r$ and
\begin{align*}
A_f(t)
= \frac{1}{t^{2n-2}} \int_{W(t)} f^*(\ddot{\omega}) \wedge \upsilon_W^{n-1}
= \int_{W(t)} f^*(\ddot{\omega}) \wedge \omega_W^{\,n-1} + A_f(0),
\end{align*}
where $n=\dim(W)$.
Here the pullback $f^*(\ddot{\omega})$ satisfies $f^*(\ddot{\omega}) = dd^c \log\left( |f_0|^2 + |f_1|^2 \right)$
for all $z$ outside of the set of indeterminacy $I_f:= \{ z \in W : f_0(z) = f_1(z) = 0 \}$ of $f$. The function $T_f$ increases in $r$ and decreases in $s$ and is of class $C^1$. If $f$ is holomorphic at $0$, then $s=0$ is permitted:
\begin{align*}
	\displaystyle
	T_f(r,0) := \int_0^r \frac{A_f(t)}{t}\, dt.
\end{align*}
Take $a\in \mathbb{P}^1=\mathbb{C}\cup\{\infty\}$ and $0<R\leq +\infty$. Let $f\not\equiv 0$ be a meromorphic function on $W(R)$.
For $0<r<R$, define the compensation of $f$ for $a$ by
\begin{align*}
m^a_f(r)=\int_{W\langle r\rangle}\log \frac{1}{||f,a||} \;\sigma_W,
\end{align*}
where $||f,a||$ denotes the chordal distance from $f$ to $a\in\mathbb{P}^1$.
Then the First Main Theorem of Nevanlinna theory becomes
$T_f(r)=T_f(r,s)=N_{\mu^a_f}(r,s)+m^a_f(r)-m^a_f(s)$,
where $0<s<r$.	
\end{defi1F}

\noindent
There is a slightly different way to continue the formulation of Nevanlinna theory from here (see \cite[pp.15]{Hu-Li-Yang-2003}).
Take $0<R\leq +\infty$. Let $f\not\equiv 0$ be a meromorphic function on $W(R)$. Let $0<s<r<R$. 
Now with the help of the positive logarithm function, we define the proximity function of $f$ by
\begin{align*}
\displaystyle m(r, f)=\int_{W\langle r\rangle} \log^+ |f|\;\sigma_W \geq  0.
\end{align*}

\noindent
The characteristic function of $f$ is defined by $T(r,f)=m(r,f)+N(r,f)$. We know that (see \cite[pp.15]{Hu-Li-Yang-2003})
\begin{align*}
\displaystyle T\left(r,\frac{1}{f}\right)=T(r,f)-\int_{W\langle s\rangle} \log |f|\;\sigma_W.
\end{align*}

\noindent
We define $m(r,a;f)=m(r,f)$ if $a=\infty$ and $m(r,a;f)=m\big(r,\frac{1}{f-a}\big)$ if $a$ is finite complex number. Now if $a\in\mathbb{C}$, then the first main theorem becomes $m(r,a;f)+N(r,a;f)=T(r,f) + O(1)$, where $O(1)$ denotes a bounded function when $r$ is sufficiently large.

\noindent
If we compare the functions $T_f(r)$ and $T(r,f)$, then (see \cite[pp.19]{Hu-Li-Yang-2003})
$T_f(r)=T(r,f)+O(1)$.

\subsection*{Growth, Exponent of convergence and Deficiency}
Consider an increasing, non-negative function $S:\mathbb{R}_{+}\to\mathbb{R}_{+}.$
The order of $S$ is defined by (see \cite[pp. 28]{Stoll-1974})
\begin{align*}
\operatorname{Ord}\, S= \limsup_{r\to\infty}\frac{\log^{+} S(r)}{\log r}.
\end{align*}

\noindent
For $0<\mu\in\mathbb{R}$, we define
\begin{align*}
t_{\mu}(S)=\limsup_{r\to\infty}\frac{S(r)}{r^{\mu}}\quad \text{and}\quad
J_{\mu}(S)=\int_{1}^{\infty}\frac{S(t)}{t^{\mu+1}}\,dt.
\end{align*}

\noindent
Clearly, if $0<\mu<\operatorname{Ord}S$, then $t_{\mu}(S)=J_{\mu}(S)=\infty$ and if $\mu>\operatorname{Ord}S$, then $t_{\mu}(S)<\infty$ and $\ J_{\mu}(S)<\infty$.
If $0<\operatorname{Ord}S=\lambda<\infty$, the function $S$ is said to be of
\begin{itemize}
	\item maximal type \quad iff \quad \(t_{\lambda}(S)=+\infty\),
	\item mean type \quad iff \quad \(0<t_{\lambda}(S)<+\infty\),
	\item minimal type \quad iff \quad \(t_{\lambda}(S)=0\).
\end{itemize}

\noindent
The function $S$ is said to be in the
\begin{itemize}
	\item convergence class \quad iff \quad \(J_{\lambda}(S)<+\infty\),
	\item divergence class \quad iff \quad \(J_{\lambda}(S)=+\infty\).
\end{itemize}

\begin{defi1G}\cite[pp. 28]{Stoll-1974}Let $\nu\geq 0$ be a non-negative divisor on $W$. Then the order, class, and type
of $\nu$ are defined as the order, class, and type of $n_{\nu}$. Observe that for each $s>0$, the order, class and type of $n_{\nu}$
coincide with the order, class and type of $N_{\nu}(\square,s)$. If $0\notin \operatorname{supp}\nu$, this is true for $s=0$ also.
\end{defi1G}

\noindent
Let $f$ be a non-constant meromorphic function in $W$. The function $f$ is said to have order $\rho(f)$, maximal, mean, minimal type or convergence class if $T(r,f)$ has those properties. Consequently the order and the lower order of $f$ are given respectively by
\[\operatorname{Odr} f= \rho(f):=\varlimsup\limits_{r\to \infty}\frac{\log^+ T(r,f)}{\log r}\;\;\text{and}\;\;\mu(f):=\varliminf\limits_{r\to \infty}\frac{\log^+ T(r,f)}{\log r}.\]

\begin{propo1A}\cite[Proposition 1.7]{Hu-Li-Yang-2003}
Suppose that $f$ and $g$ are non-constant meromorphic functions in $\mathbb{C}^n$. We have
\begin{enumerate}
    \item[(a)] $\rho(\mathcal{L}(f)) = \rho(f)$, where $\mathcal{L}(f)$ is any M\"{o}bius transformation of $f$.
    \item[(b)] $\rho (f \cdot g) \leq \max\{\rho(f), \rho(g)\}$.
    \item[(c)] $\rho(f + g) \leq \max\{\rho(f), \rho(g)\}$.
    \item[(d)] if $\rho(f) < \rho(g)$, then $\rho(fg) = \rho(f + g) = \rho(g)$.
    \item[(e)] $\max\{\mu(fg), \mu(f + g)\} \leq \max\{\rho(f), \mu(g)\}$.
    \item[(f)] if $\rho(f) < \mu(g)$, then $\max\{\mu(fg), \mu(f + g)\} \leq \mu(g)$.
    \item[(g)] if $\rho(f) < \mu(g)$, then $T(r, f) = o(T(r, g))$ as $r \to \infty$.
\end{enumerate}
\end{propo1A}

\begin{defi1H}\cite[pp. 28-29]{Stoll-1974} Let $n=\dim(W)$.
Define $A=\operatorname{supp}\nu$ and $B=A-A(1).$ A number $\mu>0$ is an exponent of convergence of $\nu$ if one-hence all-of
the following integrals exist:
\begin{align*}
&\int_{B} \nu\,\tau_{0}^{-\mu}\;\omega_W^{n-1} < \infty,\\
&\int_{B} \nu\;\tau_{0}^{\,1-n-\mu}\,v_W^{n-1} < \infty,\\
&\int_{1}^{\infty} n_{\nu}(t)\,t^{-\mu-1}\,dt < \infty,\\
&\int_{1}^{\infty} N_{\nu}(t,s)\,t^{-\mu-1}\,dt < \infty,
\quad \text{if } s>0.
\end{align*}

\noindent
Otherwise, $\mu$ is said to be an exponent of divergence of $\nu$. If $\mu>\operatorname{Ord}\nu$, then $\mu$ is an exponent of convergence of $\nu$; if $0<\mu<\operatorname{Ord}\nu$, then $\mu$ is an exponent of divergence of $\nu$.
If $0<\operatorname{Ord}\nu=\mu<\infty$, then $\mu$ is an exponent of convergence of $\nu$ if and only if $\nu$ is in the convergence class.
\end{defi1H}

\noindent
For $0\le r\le R$, we define $M_f(r)=\sup\{|f(z)|:\;||z||<r\}$, where $f$ is holomorphic. Let $s:\mathbb{R}_+\to \mathbb{R}_+$ be an increasing function and $q\in\mathbb{R}_+$. Define
\begin{align*}
K_q(r;s)=r^q\int^r_0s(t)t^{-q-1}dt+r^{q+1}\int_r^{\infty}s(t)t^{-q-2}dt\;\text{and}\;c(q)=9(q+1)^2(2+\log (1+q)).
\end{align*}

\noindent
Now we state the following result related to canonical function of finite order.
\begin{propo1B} \cite[Theorem 6.3]{Stoll-1974} Let $\nu$ be a non-negative divisor on $W$ with $0\notin \operatorname{supp}\nu$.
Assume that $q+1\in\mathbb{N}$ is an exponent of convergence of $\nu$. Let $h$ be the canonical function of $\nu$ for the constant weight $q$ of $\nu$. Take $\theta\in\mathbb{R}$ and $r\in\mathbb{R}$ with $0<\theta<1$ and $r>0$. Then
\begin{align*}
\log M_h(\theta r)\le 8n\,c(q)\,(1-\theta)^{-3n}\,K_q(r;n_\nu)\;\text{and}\;
\operatorname{Ord}\nu\le \operatorname{Ord} h\le \max\{q,\operatorname{Ord}\nu\}\le q+1.
\end{align*}

\noindent
If $q\in\mathbb{Z}_{+}$ is the smallest non-negative integer such that $q+1$ is an exponent of convergence of $\nu$, then
$q \le \operatorname{Ord}\nu = \operatorname{Ord} h \le q+1$,
$\operatorname{type} h = \operatorname{type}\nu$ provided $q < \operatorname{Ord} h \le q+1$ and
$\operatorname{class} h = \operatorname{class}\nu$ provided $q <\operatorname{Ord} h < q+1$.
\end{propo1B}

\noindent
Proposition 1.B was proved in \cite{Stoll-1953}. Let $\nu\geq 0$ be a non-negative divisor of finite order on $W$. Then the smallest non-negative integer $q$ such that $q+1$ is an exponent of convergence of $\nu$ is called the \emph{genus} of $\nu$.

\noindent
The canonical function of a non-negative divisor $\nu$ on $W$ for a constant weight $q$ can be intrinsically characterized (see \cite[Proposition 6.4]{Stoll-1974}).
\begin{propo1C}\cite[Theorem 3]{Ronkin-1968} Let $\nu\geq 0$ be a non-negative divisor on $W$ with $0\notin \operatorname{supp}\nu$.
Let $q+1\in\mathbb{N}$ be an exponent of convergence of $\nu$. Then there exists one and only one holomorphic function
$h:W\to\mathbb{C}$ such that
$h(0)=1$, $\mu_{\log h}(0)\ge q+1$, $\mu_h=\nu$, $\operatorname{Ord} h \le q+1$
and such that if $\operatorname{Ord} h=q+1$, then $h$ is of minimal type. The unique function is the canonical function of $\nu$ for the constant weight $q$.
\end{propo1C}

\noindent
Clearly, if $f$ is non-constant, then $T(r, f) \rightarrow \infty$ as $r \rightarrow$ $\infty$. Further $f$ is rational if and only if $T(r,f)=O(\log r)$ (see \cite[pp. 19]{Stoll-1974}).
On the other hand, if $f$ is transcendental, then 
\[\lim\limits_{r \rightarrow \infty} \frac{T(r, f)}{\log r}=+\infty.\] 
The deficiency of $f$ for a value $a$ is defined as $\delta_f(a)$:
\begin{align*}
0\leq \delta_f(a)=1-\limsup\limits_{r\to \infty}\frac{N(r,a;f)}{T(r,f)}\leq 1.
\end{align*}
Clearly $\delta_f(a)=1$ if $\mu^a_f$ is algebraic and if $f$ is transcendental. The set $\{a\in\mathbb{P}^1:\delta_f(a)>0\}$ is at most countable and (see \cite[pp. 20]{Stoll-1974})
\begin{align*}
\sum\limits_{a\in\mathbb{P}^1} \delta_f(a)\leq 2.
\end{align*}
An element $a\in \mathbb{P}^1$ is said to be a Picard (exceptional) value of $f$ if $a\not\in f(W)$. Obviously, $\delta_f(a)=1$ if $a$ is a Picard value of $f$.

\section{{\bf Borel exceptional value and related results in several complex variables}}

\noindent
This section constitutes the mathematical core of the paper. It establishes the main structural results on Borel exceptional values in several complex variables, presenting them in the form of clear and precise propositions. These propositions provide the essential groundwork and form the basis for the principal theorems proved in the subsequent sections. 
 
\begin{defi} Let $f$ be a transcendental meromorphic function in $\mathbb{C}^n$ with order $\rho(f)>0$. Let $a\in\mathbb{P}^1$. Then $a$ is said to be a Borel exceptional value of $f$ if $\operatorname{Ord}\mu^a_f<\rho(f)$. According to Definition 1.G, we have
\begin{align*}
\limsup\limits_{r\to \infty} \frac{\log ^+ N(r,a;f)}{\log r}<\rho(f).
\end{align*}
\end{defi}

\noindent
For a subset $E\subset [0,\infty)$, we define 
\begin{enumerate}
	\item [(i)] $m(E):=\int_E dt$, called linear measure,
	\item[(ii)] $l(E):=\int_{E\cap [1,\infty)} \frac{dt}{t}$, called logarithmic measure,
	\item[(iii)]  $	\overline{\mathrm{dens}}\;E
	=\limsup_{r\to\infty}\frac{1}{r}\int_{E\cap [1,r]} dt$, called upper density.
\end{enumerate}
It follows that if $l(E)<+\infty$ or $m(E)<+\infty$, then $E$ has zero upper density.
Let $S(f)=\{g:\mathbb{C}^n\to\mathbb{P}^1\;\text{meromorphic}: \parallel T(r,g)=o(T(r,f))\}$, where $\mathbb{P}^1=\mathbb{C}\cup\{\infty\}$ and  $\parallel$ indicates that the relation holds only outside a set $E$ with zero upper density measure and the element of $S(f)$ is called the small function of $f$.
\smallskip

\noindent
The following result is known as second main theorem:
\begin{lem2A}\label{L.2}\cite[Lemma 1.2]{Hu-Yang-1996} Let $f$ be a non-constant meromorphic function in $\mathbb{C}^n$ and let $a_1,a_2,\ldots,a_q$ be different points in $\mathbb{P}^1$. Then
\begin{align*}
(q-2)T(r,f)\leq \sideset{}{_{j=1}^{q}}{\sum} \ol N(r,a_j;f)+O(\log (rT(r,f)))
\end{align*}
holds only outside a set of finite measure on $\mathbb{R}^+$. In particular, if $f$ is of finite order $\rho(f)$, then for any $\varepsilon>0$, one has (see \cite[Corollary 1.49]{Hu-Li-Yang-2003})
\begin{align*}
(q-2)T(r,f)\leq \sideset{}{_{j=1}^{q}}{\sum} \ol N(r,a_j;f)+(\rho(f)+\varepsilon-1)\log r+O(1).
\end{align*}
\end{lem2A}

\begin{propo}\label{pro0} Let $f$ be a transcendental meromorphic function in $\mathbb{C}^n$ with finite positive order $\rho(f)$. Then $f$ has at most two Borel exceptional values.
\end{propo}
\begin{proof} If possible suppose there are at least three distinct Borel exceptional values $a_1$, $a_2$ and $a_3$, say. Let 
\begin{align}\label{pro1:1.1}
 \limsup\limits_{r\to \infty} \frac{\log ^+ N(r,a_j;f)}{\log r}=s_j
 \end{align}
and $s=\max\{s_1,s_2,s_3\}$. Clearly $s<\rho(f)$. Let us choose $\varepsilon>0$ in such a way that $\rho(f)+\varepsilon-1>0$.  Now using Lemma 2.A, we get
\begin{align}\label{pro1:1.2}
T(r,f)\leq \sideset{}{_{j=1}^{3}}{\sum} \ol N(r,a_j;f)+(\rho(f)+\varepsilon-1)\log r+O(1).
\end{align}
Taking $\log^+$ on the both sides of (\ref{pro1:1.2}) and then using (\ref{pro1:1.1}), we get
\begin{align*}
\rho(f)=\limsup\limits_{r\to \infty} \frac{\log ^+ T(r,f)}{\log r}\leq s<\rho(f),
\end{align*}
which is impossible. Hence $f$ has at most two Borel exceptional values.
\end{proof}

\begin{rem} Obviously, when $\rho(f)>0$, every Picard exceptional value is the Borel exceptional. For example,  $1$ is the Borel exceptional value of $f(z)=e^{z_1+z_2+\ldots+z_n}+1$.
\end{rem}

\begin{lem2B}\label{2A}\cite[Lemma 3.59]{Hu-Li-Yang-2003} Let $P$ be a non-constant entire function in $\mathbb{C}^n$. Then 
\begin{align*}
\rho(e^P)=
\begin{cases}
\deg(P), & \text{if $P$ is a polynomial,}\\
+\infty, & \text{otherwise.}
\end{cases}
\end{align*}
\end{lem2B}

\begin{propo}\label{Pro1} Let $f$ be a transcendental meromorphic function in $\mathbb{C}^n$ such that $f(0)\neq 0,\infty$ and $\rho(f)$ be finite. Then $$
f=\frac{H_1}{H_2}e^{\alpha},$$
where $H_1$ and $H_2$ are canonical functions such that $\mu^0_f=\mu^0_{H_1}$, $\mu^{\infty}_f=\mu^0_{H_2}$,
\begin{align*}
\rho(H_1)=\limsup\limits_{r\to\infty} \frac{\log^+ N_{\mu^0_{f}}(r,0)}{\log r}\leq \rho(f),\;\;
\rho(H_2)=\limsup\limits_{r\to\infty} \frac{\log^+ N_{\mu^{\infty}_{f}}(r,0)}{\log r}\leq \rho(f)
\end{align*}
and $\alpha$ is a polynomial in $\mathbb{C}^n$ such that $\deg(\alpha)=\rho(e^{\alpha})\leq \rho(f)$.
\end{propo}
\begin{proof} We note that for every meromorphic function $f$ in $\mathbb{C}^n$ there are entire functions $g$ and 
$h$ such that 
\begin{align}\label{p1.1}
f=g/h
\end{align}
in $\mathbb{C}^n$ and $g$ and $h$ are coprime at every point of $\mathbb{C}^n$ \cite[pp. 186]{LH1} which implies 
that $F=(h,g):\mathbb{C}^n\to \mathbb{C}^2$ is a reduced representative of $f$. Since $f(0)\neq 0, \infty$, we have $g(0)\neq 0$ and $h(0)\neq 0$. Let $\nu=\mu^0_g$. Clearly $\nu$ is a non-negative divisor on $\mathbb{C}^n$ such that $0\not\in \text{supp}\;\nu$. Suppose $q\in\mathbb{Z}_+$ is the smallest non-negative integer such that $q+1$ is an exponent of convergence of $\nu$. According to Propositions 1.B and 1.C, there exists one and only one holomorphic function $H_1:\mathbb{C}^n\to \mathbb{C}$ such that $\nu=\mu^0_{H_1}$ and 
\begin{align*}
\rho(H_1)=\text{Ord}\;\nu=\limsup\limits_{r\to\infty} \frac{\log^+ N_{\mu^0_{H_1}}(r,0)}{\log r}=\limsup\limits_{r\to\infty} \frac{\log^+ N_{\mu^0_{g}}(r,0)}{\log r}=\limsup\limits_{r\to\infty} \frac{\log^+ N_{\mu^0_{f}}(r,0)}{\log r}.
\end{align*}
Now using first main theorem, we deduce that
\begin{align}\label{p1.2}
\rho(H_1)=\limsup\limits_{r\to\infty} \frac{\log^+ N_{\mu^0_{f}}(r,0)}{\log r}\leq \limsup\limits_{r\to\infty} \frac{\log^+ T(r,f)}{\log r}=\rho(f).
\end{align}
Note that $g$ and $H_1$ are entire functions in $\mathbb{C}^n$ with the same divisor $\mu^0_g=\mu^0_{H_1}$. Then there exists an entire function $\alpha_1$ in $\mathbb{C}^n$ such that $g=H_1e^{\alpha_1}$.

\noindent
Again let $\nu=\mu^0_h$. Clearly $\nu$ is a non-negative divisor on $\mathbb{C}^n$ such that $0\not\in \text{supp}\;\nu$. Then there exists one and only one holomorphic function $H_2:\mathbb{C}^n\to \mathbb{C}$ such that $\nu=\mu^0_{H_2}$ and 
\begin{align}\label{p1.3}
\rho(H_2)=\limsup\limits_{r\to\infty} \frac{\log^+ N_{\mu^{\infty}_{f}}(r,0)}{\log r}\leq \limsup\limits_{r\to\infty} \frac{\log^+ T(r,f)}{\log r}=\rho(f).
\end{align}
Since $h$ and $H_2$ are entire functions in $\mathbb{C}^n$ with the same divisor $\mu^0_h=\mu^0_{H_2}$, there exists an entire function $\alpha_2$ in $\mathbb{C}^n$ such that $h=H_2e^{\alpha_2}$. If we take $\alpha=\alpha_1-\alpha_2$, then from (\ref{p1.1}), we have
\begin{align*}
f=\frac{H_1}{H_2}e^{\alpha},
\end{align*}
where $H_1$ and $H_2$ are canonical functions such that $\mu^0_f=\mu^0_{H_1}$, $\mu^{\infty}_f=\mu^0_{H_2}$, (\ref{p1.2}) and (\ref{p1.3}) hold and $\alpha$ is an entire function in $\mathbb{C}^n$ such that $\rho(e^{\alpha})\leq \rho(f)$. If $\alpha$ is a constant, then $\deg(\alpha)=\rho(e^{\alpha})\leq \rho(f)$ holds trivially. Next suppose $\alpha$ is non-constant. Since $\rho(f)$ is finite $f.\frac{H_2}{H_1}=e^{\alpha}$, according to Lemma 2.B, $\alpha$ must be a polynomial in $\mathbb{C}^n$ such that $\deg(\alpha)=\rho(e^{\alpha})\leq \rho(f)$.
\end{proof}

\begin{propo}\label{Pro2} Let $h$ be a non-constant polynomial in $\mathbb{C}^n$ and let $f=e^{h}$. Then $\rho(f)=\mu(f)=\deg(h)$.
\end{propo}
\begin{proof} Note that $|h(z)|\leq C||z||^{\deg(h)}$ holds for $||z||>1$, where $C$ is a positive constant. Now for sufficiently large value of $r$, we have $T(r,f)=m\left(r,e^h\right)=O\left(r^{\deg(h)}\right)$. Therefore it is easy to verify that $\rho(f)\leq \deg(h)$ and $\mu(f)\leq \deg(h)$. Take $\xi\in\mathbb{C}^n\backslash \{0\}$. Define a holomorphic mapping $j_{\xi}:\mathbb{C}\to \mathbb{C}^n$ by $j_{\xi}(\zeta)=\zeta \xi$. Clearly
\begin{align*}
f_{\xi}=f\circ j_{\xi}=e^{h\circ j_{\xi}}=e^{h_{\xi}}:\mathbb{C}\to \mathbb{C}
\end{align*}
is a holomorphic mapping. We can choose $\xi\in\mathbb{C}^n[0,1]\backslash \{0\}$ such that $h_{\xi}$ is a  polynomial in $\mathbb{C}$ and $\deg(h_{\xi})=\deg(h)$. But in $\mathbb{C}$, we know that (see Theorem 1.44 \cite{Yang-Yi-2003}) $\rho(f_{\xi})=\mu(f_{\xi})=\deg(h_{\xi})=\deg(h)$. On the other hand, we know that (see \cite[pp. 286]{Hu-Li-Yang-2003})
\begin{align}\label{pro2:1.1} 
T_{f_{\xi}}(r,0)\leq \frac{1+\theta}{(1-\theta)^{2n-1}} T_{f}\left(\frac{r}{\theta},0\right),\quad \xi\in\mathbb{C}^n[0,1]\backslash \{0\}
\end{align}
where $0<\theta<1$. Now inequality (\ref{pro2:1.1}) implies that
$\rho(f_{\xi})\leq \rho(f)$ and $\mu(f_{\xi})\leq \mu(f)$. Consequently, we have $\deg(h)\leq \rho(f)\leq \deg(h)$ and $\deg(h)\leq \mu(f)\leq \deg(h)$. Finally, we have $\mu(f)=\rho(f)=\deg(h)$.
\end{proof}

\begin{propo}\label{Pro3} Let $f$ be a transcendental meromorphic function in $\mathbb{C}^n$ and $\rho(f)>0$ be finite. If $f$ has two distinct Borel exceptional value, say $a_1$ and $a_2$ such that $f(0)\neq a_1,a_2$ then $\rho(f)=\mu(f)\in\mathbb{N}$.
\end{propo}
\begin{proof} Without loss of generality, we assume that $a_1=0$ and $a_2=\infty$ otherwise it is enough to study the function $\frac{f-a_1}{f-a_2}$. According to Proposition \ref{Pro1}, we have
\begin{align}\label{Pro3:1.1}
f=\frac{H_1}{H_2}e^{P},
\end{align}
where $H_1$ and $H_2$ are canonical functions such that
\begin{align}\label{Pro3:1.2}
\rho(H_1)=\limsup\limits_{r\to \infty}\frac{\log^+ N_{\mu^0_{H_1}}(r,0)}{\log r}=\limsup\limits_{r\to \infty}\frac{\log^+ N_{\mu^0_{f}}(r,0)}{\log r}<\rho(f),
\end{align}
\begin{align}\label{Pro3:1.3}
\rho(H_2)=\limsup\limits_{r\to \infty}\frac{\log^+ N_{\mu^0_{H_2}}(r,0)}{\log r}=\limsup\limits_{r\to \infty}\frac{\log^+ N_{\mu^{\infty}_{f}}(r,0)}{\log r}<\rho(f)
\end{align}
and $P$ is a polynomial in $\mathbb{C}^n$ such that $\deg(P)=\rho(e^P)\leq \rho(f)$. If we take 
\begin{align}\label{Pro3:1.4}
H=H_1/H_2,
\end{align}
then from (\ref{Pro3:1.1}), we have 
\begin{align}\label{Pro3:1.5}
e^P=f/H.
\end{align}
Now by Proposition \ref{Pro2}, we have $\mu(e^P)=\rho(e^P)=\deg(P)$. Note that by Proposition 1.A(a), we have $\rho(1/H)=\rho(H)$.
Again by Proposition 1.A(b), we have from (\ref{Pro3:1.2})-(\ref{Pro3:1.4}) that
\begin{align*}
\rho(H)\leq \max\{\rho(H_1),\rho(H_2)\}<\rho(f).
\end{align*}
and so from (\ref{Pro3:1.5}), we have
\begin{align*}
\rho(e^P)\leq \max\{\rho(H),\rho(f)\}=\rho(f).
\end{align*}
Again from (\ref{Pro3:1.5}), we get $f=He^P$ and so by Proposition 1.A(b), we have
\begin{align*}
\rho(f)\leq \max\{\rho(e^P),\rho(H)\}=\rho(e^P).
\end{align*}
Therefore $\rho(f)=\rho(e^P)$ and so $\rho(H)<\mu(e^P)$. Since $f=He^P$, using Proposition 1.A(e), we get $\mu(f)\leq \max\{\rho(H),\mu(e^P)\}=\mu(e^P)$. 
Since $e^P=f.\frac{1}{H}$, using Proposition 1.A(e), we have 
\begin{align*}
\mu(e^P)\leq \max\{\rho(H),\mu(f)\}\leq \max\{\rho(H),\mu(e^P)\}=\mu(e^P),
\end{align*}
i.e., $\max\{\rho(H),\mu(f)\}=\mu(e^P)$. As $\rho(H)<\mu(e^P)$, it follows that $\mu(f)=\mu(e^P)=\rho(e^P)=\rho(f)$. Therefore $\mu(f)=\rho(f)=\rho(e^P)=\deg(P)\in\mathbb{N}$.
\end{proof}

\begin{propo}\label{Pro4} Let $f$ be a transcendental meromorphic function in $\mathbb{C}^n$ and $\rho(f)>0$ be finite. If $f$ has two distinct Borel exceptional value, say $a_1$ and $a_2$ such that $f(0)\neq a_1,a_2$, then $\delta(a_1;f)=\delta(a_2;f)=1.$
\end{propo}
\begin{proof} Without loss of generality, we assume that $a_1=0$ and $a_2=\infty$ otherwise it enough to study the function $\frac{f-a_1}{f-a_2}$. According to Proposition \ref{Pro1}, we have
\begin{align}\label{Pro4:1.1}
f=\frac{H_1}{H_2}e^{P},
\end{align}
where $H_1$ and $H_2$ are canonical functions such that (\ref{Pro3:1.2}) and (\ref{Pro3:1.3}) hold and $P$ is a polynomial in $\mathbb{C}^n$ such that $\deg(P)=\rho(e^P)\leq \rho(f)$. Now by proposition \ref{Pro3}, we have $\mu(f)=\rho(f)$. If we take $H=H_1/H_2$, then using Proposition 1.A(b), we have from (\ref{Pro3:1.2})-(\ref{Pro3:1.4}) that
\begin{align*}
\rho(H)\leq \max\{\rho(H_1),\rho(H_2)\}<\rho(f)=\mu(f).
\end{align*}
Again using Proposition 1.A(g), we conclude that $T(r,H)=o(T(r,f))$. Now in view of first main theorem and using (\ref{Pro3:1.2}) and (\ref{Pro3:1.3}), we have respectively $N(r,0;f)=o(T(r,f))$ and $N(r,\infty;f)=o(T(r,f))$. Consequently $\delta(a_1;f)=\delta(a_2;f)=1$.
\end{proof}

%\begin{table}[htbp]
%\centering
%\renewcommand{\arraystretch}{1.2}
	
%\begin{tabularx}{\textwidth}{||>{\centering\arraybackslash}p{3.2cm}||
%			>{\centering\arraybackslash}p{4.8cm}||
%			>{\centering\arraybackslash}p{6.0cm}||}
%		\hline\hline
%		\textbf{Result} & \textbf{Main Tool Used} & \textbf{Purpose} \\
%		\hline\hline
%		Lemma~\protect\ref{L.2} & Second Main Theorem & Growth comparison \\
%		\hline
%		Proposition~\protect\ref{pro0} & Lemma~\protect\ref{L.2} & Maximum two exceptional values \\
%		\hline
%		Proposition~\protect\ref{Pro1} & Canonical representation & Structural form of $f$ \\
%		\hline
%		Proposition~\protect\ref{Pro2} & One-variable reduction & Growth of $e^h$ \\
%		\hline
%		Proposition~\protect\ref{Pro3} & Propositions~\protect\ref{Pro1}, \protect\ref{Pro2} & $\rho(f)=\mu(f)$ \\
%		\hline
%		Proposition~\protect\ref{Pro4} & Proposition~\protect\ref{Pro3} & Deficiency equals one \\
%		\hline\hline
%	\end{tabularx}
%	\vspace{.5cc}
%\caption{Summary of results, tools and purposes}
%\end{table}

\section{{\bf Shared Value Problems Associated with Shifts}}

\noindent
Let $f$, $g$, and $a$ be meromorphic functions on $\mathbb{C}^n$. Then there exist pairs of entire functions $(f_1,f_2)$, $(g_1,g_2)$, and $(a_1,a_2)$, each coprime at every point in $\mathbb{C}^n$, such that $f=f_2/f_1$, $g=g_2/g_1$ and $a=a_2/a_1$.
We say that $f$ and $g$ \emph{share $a$ counting multiplicities (CM)} if
\[
\mu^0_{a_1f_2-a_2f_1}
=\mu^0_{a_1g_2-a_2g_1}
\quad (a\not\equiv\infty),\;\;\text{and}\;\;
\mu^0_{f_1}=\mu^0_{g_1}
\quad (a=\infty).
\]
Similarly, $f$ and $g$ are said to \emph{share $a$ ignoring multiplicities (IM)} if
\[
\mu^0_{a_1f_2-a_2f_1,1}
=\mu^0_{a_1g_2-a_2g_1,1}
\quad (a\not\equiv\infty),\;\;\text{and}\;\;
\mu^0_{f_1,1}=\mu^0_{g_1,1}
\quad (a=\infty).
\]
\vspace{1.2mm}
%The classical Nevanlinna value distribution theory, introduced in the 1920s by Rolf Nevanlinna, studies the manner in which meromorphic functions of a single complex variable assume their values. It provides profound insights into the growth, uniqueness, and value-sharing properties of meromorphic functions and has become a cornerstone of complex analysis.

\noindent
With the development of complex analysis in higher dimensions, mathematicians sought to extend Nevanlinna's foundational results to functions of several complex variables. This endeavor led to the emergence of Nevanlinna theory in several complex variables, which investigates the growth and value distribution of holomorphic and meromorphic mappings from $\mathbb{C}^n$ into complex manifolds and complex projective spaces. 
In recent years, Nevanlinna theory in several complex variables has evolved into a dynamic and rapidly advancing field of research (see \cite{BM}-\cite{Banerjee-Majumder-Pramanik-Sarkar-2026}, \cite{ck}-\cite{CLL}, \cite{PVD2}-\cite{HZ1}, \cite{Korhonen-2012}-\cite{LS1}, \cite{FL1}-\cite{Majumder-Sarkar-2027}, \cite{GS1}, \cite{GS2}, \cite{Wang-Liu-2023}-\cite{XuHaldar2024}). 
\smallskip

\noindent
A particularly active area concerns the application of value distribution theory to problems involving normal families, partial differential equations (PDEs), partial difference equations, and partial differential-difference equations. 
\smallskip

\noindent
In 2012, Korhonen \cite{Korhonen-2012} established the first difference analogue of the logarithmic derivative lemma, known as the \emph{logarithmic difference lemma}, for meromorphic functions in $\mathbb{C}^m$ with hyper-order strictly less than $\frac{2}{3}$. Subsequently, Cao and Korhonen \cite{ck} extended this result to the case where the hyper-order is strictly less than $1$. These foundational results play a crucial role in the study of meromorphic solutions of complex difference equations and in the uniqueness theory of difference operators in higher dimensions.

\noindent
We begin by recalling a fundamental result of Heittokangas et al. \cite{Heittokangas-Korhonen-Laine-Rieppo-Zhang-2009}, which establishes a strong structural behavior of meromorphic functions of small order. 
\begin{theo3A} \cite[Theorem 1]{Heittokangas-Korhonen-Laine-Rieppo-Zhang-2009} Let $f$ be a non-constant meromorphic function in $\mathbb{C}$ such that $\rho(f)<2$ and let $c$ be a non-zero complex number. If $f(z+c)$ and $f(z)$ share $a\in\mathbb{C}$ and $\infty$ CM, then $f(z+c)-a=d(f(z)-a)$ for all $z\in\mathbb{C}$, where $d$ is some non-zero complex number.
\end{theo3A}

\noindent
The following example demonstrates that the order condition ``$\rho(f)<2$'' of Theorem 3.A is indispensable.
\begin{exm3A}\cite{Heittokangas-Korhonen-Laine-Rieppo-Zhang-2009} Let $f(z)=e^{z^{2}} + 1$, and let $c$ be a non-zero finite complex number. Then it immediately follows that $f(z+c)-1=(f(z)-1)e^{2zc+c^{2}}$ for all $z \in \mathbb{C}$. Moreover, we find that
$N(r,1;f)=0$, and hence $1$ is a Borel exceptional value of $f$.
\end{exm3A}

\noindent
Motivated by the role of Borel exceptional values, for entire functions of finite order, Huang and Zhang \cite{Huang-Zhang-2018} obtained the following refinement. 
\begin{theo3B}\cite[Theorem 2.3]{Huang-Zhang-2018} Let $f(z)$ be a finite-order transcendental entire function in $\mathbb{C}$ with a Borel exceptional value $\alpha \in \mathbb{C}$, let $c \in \mathbb{C}$, and let $a(z) \,(\not\equiv \alpha) \in \mathcal{S}(f)$.
If $f(z+c)$ and $f(z)$ share $a(z)$ CM, then $f(z+c) \equiv f(z)$.
\end{theo3B}

\noindent
In the same spirit, Huang and Zhang also established a corresponding result involving higher-order difference operators.
\begin{theo3C}\cite[Theorem 2.6.]{Huang-Zhang-2018} Let $f$ be a transcendental entire function in $\mathbb{C}$ such that $\rho(f)<2$.
If $\Delta_c^{k} f(z)$ and $f(z)$ share $0$ CM, where $k \in \mathbb{N}$ and
$c\in\mathbb{C}\backslash \{0\}$ are such that $\Delta_c^{k} f(z) \not\equiv 0$, then 
$\Delta_c^{k} f(z) \equiv df(z)$ for all $z\in\mathbb{C}$, where $d$ is some non-zero complex number.
\end{theo3C}

\noindent
Recently, Majumder and Das \cite{Majumder-Das-2025} extended Theorem 3.A to functions of several complex variables.
\begin{theo3D}\cite[Thorem 1.4]{Majumder-Das-2025} Let $f$ be a non-constant meromorphic function in $\mathbb{C}^n$ such that $\rho(f)<2$ and let $c\in\mathbb{C}^n\backslash \{0\}$. If $f(z+c)$ and $f(z)$ share $a\in\mathbb{C}$ and $\infty$ CM, then $f(z+c)-a=d(f(z)-a)$ for all $z\in\mathbb{C}^n$, where $d$ is some non-zero complex number.
\end{theo3D}

\noindent
The following example confirms that the condition ``$\rho(f)<2$'' remains necessary in higher dimensions.
\begin{exm}\cite{Majumder-Das-2025}
Let $f(z)=f(z_1,z_2)=e^{z_1^2+z_2^2}+1$ and $c=(c_1,c_2)=(1,0)$. Clearly $f(z+c)$ and $f(z)$ share $1$ and $\infty$ CM, but $f(z+c)-1=e^{2z_1+1}(f(z)-1)$.
\end{exm}

\noindent
We now present our first main contribution of this section.
\begin{theo}\label{t2.2} Let $f$ be a transcendental meromorphic function in $\mathbb{C}^n$ such that $\rho(f)<2$.
If $\Delta_c^{k} f(z)$ and $f(z)$ share $0$ and $\infty$ CM, where $k \in \mathbb{N}$ and
$c\in\mathbb{C}^n \backslash \{0\}$ are such that $\Delta_c^{k} f(z) \not\equiv 0$, then $\rho(f)\geq 1$ and
$\Delta_c^{k} f(z) \equiv df(z)$ for all $z\in\mathbb{C}^n$, where $d$ is some non-zero complex number.
\end{theo}

\begin{rem} Obviously Theorem \ref{t2.2} extends Theorem 3.C to the case of higher dimensions.
\end{rem}

\noindent
The following examples show that conclusion $\Delta_c^{k} f(z) \equiv df(z)$ for all $z\in \mathbb{C}^n$
of Theorem \ref{t2.2}.

\begin{exm} Let $f:\mathbb{C}^n\rightarrow \mathbb{P}^1$ be defined by 
\[f(z)=\frac{e^{\alpha (z_1+z_2+\cdots+z_n)}}{e^{\beta (z_1+z_2+\cdots+z_n)}+1},\]
where $\alpha,\;\beta\in \mathbb{C}\backslash \{0\}$ such that $e^{n\alpha}=3$, $e^{n\beta}=1$ and $c=(1,1,\ldots,1)$. Note that 
\[\Delta_cf(z)=\frac{2e^{\alpha (z_1+z_2+\cdots+z_n)}}{e^{\beta (z_1+z_2+\cdots+z_n)}+1}.\]
Clearly $\Delta_cf(z)=2f(z)$ and so $\Delta_cf$ and $f$ share $0$ and $\infty$ CM.
\end{exm} 
\begin{exm}
Let $f:\mathbb{C}^n \to \mathbb{P}^1$ be defined by
\[f(z)=\frac{(e^{2\pi i \lambda/\mu} - 1) e^{\lambda (z_1+\cdots+z_n)}}{e^{\mu (z_1+\cdots+z_n)} - 1},\] $\lambda,\mu \in \mathbb{C}\setminus\{0\}$, $\frac{\lambda}{\mu}\not\in\mathbb{Z}$ and let $c = \left(\frac{2\pi i}{\mu},0,\dots,0\right) \in \mathbb{C}^n$. Then \[\Delta_c^k f(z) = \frac{(e^{2\pi i \lambda/\mu} - 1)^{k+1}\, e^{\lambda (z_1+\cdots+z_n)}}{e^{\mu (z_1+\cdots+z_n)} - 1}.\]	
Here, $\Delta_c^k f(z) \equiv d\, f(z)$ with $d = (e^{2\pi i \lambda/\mu} -1)^{k+1} \neq 0$.
Clearly: $\Delta_c^k f$  and $f$ share $0$ and $\infty$ CM and $\rho(f)=1<2$.
\end{exm}

\noindent
The following examples demonstrate that even when $f(z)$ and $\Delta_c^{k} f(z)$ share a non-zero constant $a$ CM, it does not necessarily follow that
$\Delta_c^{k} f(z) \equiv df(z)$ for all $z\in \mathbb{C}^n$, where $d$ is a non-zero complex number.

\begin{exm} Let \[f(z)=e^{z_1+z_2+\cdots+z_n}+e^{-(z_1+z_2+\cdots+z_n)}+2,\] $c=\left(\frac{\pi i}{n},\ldots,\frac{\pi i}{n}\right)$. Note that \[\Delta_cf(z)=-2(e^{z_1+z_2+\cdots+z_n}+e^{-(z_1+z_2+\cdots+z_n)}).\]
Clearly $\Delta_cf$ and $f$ share $\frac{4}{3}$ and $\infty$ CM, but $\Delta_cf\not\equiv df$.
\end{exm}
 \begin{exm}
 Let \[f(z) = \frac{e^{(z_1+\cdots+z_n)} }{e^{2 (z_1+\cdots+z_n)} +1}+1,\] $c=\left(\frac{\pi i}{n},\ldots,\frac{\pi i}{n}\right)$. Note that \[\Delta_c^k f(z) = (-1)^{k}2^{k}\frac{e^{(z_1+\cdots+z_n)} }{e^{2 (z_1+\cdots+z_n)} +1}.\]	
 Clearly, $\Delta_c^k f(z)$ and $f(z)$ share $a = \frac{1}{1+(-1)^{k+1}2^{-k}}$ and $\infty$ CM but $\Delta_c^kf\not\equiv df$.
 \end{exm}
 
 \smallskip
 
 \noindent
 The following result extends Theorem 3.D to a broader setting.
\begin{theo}\label{t2.3} Let $f$ be a finite-order transcendental entire function in $\mathbb{C}^n$ with a Borel exceptional value $\alpha \in \mathbb{C}$ such that $f(0)\neq \alpha$, $c \in \mathbb{C}^n\backslash \{0\}$ and $a(\not\equiv \alpha)$ be an entire function such that $a\in \mathcal{S}(f)$.
If $f(z+c)$ and $f(z)$ share $a(z)$ CM, then $f(z+c) \equiv f(z)$.
\end{theo}

 \noindent
In the following example, we see that the condition ``$a(\not\equiv \alpha)$'' in Theorem \ref{t2.3} is necessary.

\begin{exm} Let \[f(z)=e^{z_1+z_2+z_3}(z_2+z_3+1)(z_1-z_2+1)+1\] and $c=\left(\frac{\pi \iota}{2},\frac{\pi \iota}{2},-\frac{\pi\iota}{2}\right)$. Note that $f(z+c)=\iota e^{z_1+z_2+z_3}(z_2+z_3+1)(z_1-z_2+1)+1$. Clearly $1$ is a Borel exceptional value of $f$ and $f(0)\neq 1$. It is easy to verify that $f(z+c)$ and $f(z)$ share $1$ CM, but $f(z+c)\not\equiv f(z)$.
\end{exm}

 \noindent
The following example demonstrates that Theorem \ref{t2.3} fails to hold when $f$ is a transcendental meromorphic function.
\begin{exm} Let $f:\mathbb{C}^n\rightarrow \mathbb{P}^1$ be defined by 
\[f(z)= \frac{e^{(z_1+z_2+\cdots+z_n)}-1}{e^{(z_1+z_2+\cdots+z_n)}+1}\] and let $c=\left(c_1,c_2,\cdots,c_n\right)$ such that $e^{(c_1+c_2+\cdots+c_n)}=-1$. Clearly $1$ is a Borel exceptional value of $f$ and $f(0)\neq 1$. It is easy to very that $f(z+c)$ and $f(z)$ share $-1$ CM, but $f(z+c)\not\equiv f(z)$.
\end{exm}

% \noindent
%Following example shows that the condition ``$\rho(f)<+\infty$'' in Theorem \ref{t2.3} is necessary.
%\begin{exm} Let $f:\mathbb{C}^n\rightarrow \mathbb{P}^1$ be defined by 
%\[\displaystyle f(z)=a\left(1-e^{-\sin(z_1+z_2+\cdots+z_n)}\right)\]
%and $c=(c_1,c_2,\ldots,c_n)=(r_1\pi,r_2\pi,\ldots,r_n\pi)$, where $\sum_{i=1}^nr_i$ is an odd positive integer, a is a non-zero constant. Note that
%\[\displaystyle f(z+c)=a\left(1-e^{\sin(z_1+z_2+\cdots+z_n)}\right).\]
%Clearly $a$ is a Borel exceptional value of $f$ and $f(0)\neq a $. It is easy to verify that $f(z+c)$ and $f(z)$ share $2a$ CM, but $f(z+c)\not\equiv f(z)$.
%\end{exm}
\vspace{1.8mm}

\noindent
Next we consider the system of complex partial differential-difference equations
\begin{align}\label{Eq3.1}
\frac{\partial f(z)}{\partial z_j}=f(z+c),
\end{align}
where $j\in\{1,2,\ldots,n\}$ and $c\in\mathbb{C}^n\backslash \{0\}$.
The solutions of (\ref{Eq3.1}) exist, for example, $f(z)=e^{z_1+z_2+\ldots+z_n}$ is a solution of (\ref{Eq3.1}), where
$c =(2k\pi \iota,2k\pi \iota,\ldots,2k\pi \iota)$, $k\in\mathbb{Z}$ and $f(z)=\sin (z_1+z_2+\ldots+z_n)$ or $f(z)=\cos (z_1+z_2+\ldots+z_n)$ are also solutions of (\ref{Eq3.1}) for suitable $c\in\mathbb{C}^n$. Obviously, the equation
(\ref{Eq3.1}) has no rational solutions in $\mathbb{C}^n$.
\smallskip

\noindent
In this section, we establish our results concerning the system of complex partial differential-difference equations \eqref{Eq3.1}:
\begin{theo}\label{t2.4} The order of meromorphic solutions in $\mathbb{C}^n$ of (\ref{Eq3.1}) must satisfy $\rho(f)\geq 1$.
\end{theo}

\begin{theo}\label{t2.5} Let $f$ be a finite order entire solution in $\mathbb{C}^n$ to the equation (\ref{Eq3.1}) such that $f(0)\neq 0$ and $0$ is a Borel exceptional value of $f$. Then, $\rho(f)=\mu(f)\in\mathbb{N}$ and one of the following cases holds:
\begin{enumerate}
\item[(i)] $f(z)=g(z)e^{a_1z_1+\ldots+a_nz_n+Q(z)}$, where $a_1,\ldots,a_n\in\mathbb{C}$, $g$ is a polynomial in $\mathbb{C}^n$ and $ Q(z_1,\ldots,z_{j-1},z_{j+1},\ldots,z_n)$ is a polynomial in $\mathbb{C}^{n-1}$ such that $Q(z_1+c_1,\ldots,z_{j-1}+c_{j-1},z_{j+1}+c_{j+1},\ldots,z_n+c_n)\equiv  Q(z_1,\ldots,z_{j-1},z_{j+1},\ldots,z_n)$, $a_j=e^{a_1c_1+\ldots+a_nc_n}$ and
\begin{align*}
	\frac{\partial g(z)}{\partial z_j}=a_j(g(z+c)-g(z));
\end{align*}

\item[(ii)] $f(z)=g(z)e^{a_1z_1+\ldots+a_nz_n+Q(z)}$, where $a_1,\ldots,a_n\in\mathbb{C}$, $g$ is a transcendental function in $\mathbb{C}^n$ and $ Q$ is a polynomial in $\mathbb{C}^{n}$ such that $Q(z+c)\equiv  Q(z)$, $\rho(g)<\rho(f)$ and 
\[\frac{\partial g(z)}{\partial z_j}+\left(a_j+\frac{\partial Q(z)}{\partial z_j}\right)g(z)=e^{a_1c_1+a_2c_2+\ldots+a_nc_n}g(z+c).\]
\end{enumerate}
\end{theo}

\begin{rem} Obviously Theorem \ref{t2.5} extends Theorem 3.5 \cite{Liu-Dong-2014} to the case of higher dimensions.
\end{rem}

\noindent
The following examples demonstrate the existence of the solutions described in conclusion (i) of Theorem \ref{t2.5}.
\begin{exm}
Let $f(z_1,z_2)=g(z)e^{a_1z_1+a_2z_2+Q(z_2)}$, where $a_1=1$, $a_2=0$, $g(z)\equiv 1$ and $Q(z_2)=z_2^m$, $m\in\mathbb{N}\backslash\{1\}$ and $c=(c_1,c_2)=(2\pi i,0)$. Clearly $f(z_1,z_2)$ is an entire function in $\mathbb{C}^2$, $f(0,0)=1\neq 0$, $0$ is a Borel exceptional value of $f(z_1,z_2)$, $Q(z_2+c_2)=Q(z_2)$ and 
\begin{align*}
	\frac{\partial f(z_1,z_2)}{\partial z_1}=e^{z_1+z_2^m}=f(z+c).
\end{align*}
\end{exm}

\begin{exm}
	Let $f(z_1,z_2,z_3)=g(z)e^{a_1z_1+a_2z_2+a_3z_3+Q(z_2,z_3)}$, where $a_1=1$, $a_2=0$, $a_3=0$, $g(z)\equiv 1$ and $Q(z_2,z_3)=(z_2-z_3)^m$, $m\in\mathbb{N}$ and $c=(c_1,c_2,c_3)=(2\pi i,1,1)$. Clearly $f(z_1,z_2,z_3)$ is an entire function in $\mathbb{C}^3$, $f(0,0,0)=1\neq 0$, $0$ is a Borel exceptional value of $f(z_1,z_2,z_3)$, $Q(z_2+c_2,z_3+c_3)=Q(z_2,z_3)$ and 
	\begin{align*}
		\frac{\partial f(z_1,z_2,z_3)}{\partial z_1}=e^{z_1+(z_2-z_3)^m}=f(z+c).
	\end{align*}
\end{exm}
\begin{exm}
	Let $f(z_1,z_2)=g(z)e^{a_1z_1+a_2z_2+Q(z_2)}$, where $a_1=1$, $a_2=-1$, $g(z)=1+\frac{1}{2}z_1^2-\frac{1}{2}z_2$ and $Q(z_2)\equiv 0$ and $c=(c_1,c_2)=(1,1)$. Clearly $f(z_1,z_2)$ is an entire function in $\mathbb{C}^2$, $f(0,0)=1\neq 0$, $0$ is a Borel exceptional value of $f(z_1,z_2)$, $Q(z_2+c_2)=Q(z_2)$, $\frac{\partial g(z_1,z_2)}{\partial z_1}=z_1$ and 
	\[g(z+c)-g(z)=\frac{1}{2}((z_1+1)^2-z_1^2)-\frac{1}{2}((z_2+1)-z_2)=z_1.\]
	Also it is easy to verify that $f(z)$ satisfies the following equation
	\begin{align*}
		\frac{\partial f(z_1,z_2)}{\partial z_1}=f(z+c)
	\end{align*}
	whereas $g(z)$ satisfies the following equation
	\begin{align*}
		\frac{\partial g(z)}{\partial z_1}=a_1(g(z+c)-g(z)).
	\end{align*}
\end{exm}

\subsection{{\bf Auxiliary lemmas}}
\begin{lem}\label{L2.1}\cite[Theorem 2.4]{Cao-Xu-2020} Let $f$ be a non-constant meromorphic function on $\mathbb{C}^n$ and let $c \in \mathbb{C}^n\setminus \{0\} $. If $f$ is of finite order, then 
\begin{align*}
 m\left(r,\frac{f(z+c)}{f(z)}\right)+m\left(r, \frac{f(z)}{f(z+c)}\right)=O\left(r^{\rho(f)-1+\varepsilon}\right)
 \end{align*}
holds for any $\varepsilon>0$.
\end{lem}

\begin{lem}\label{L2.2}\cite[Theorem 3.3]{Cao-Xu-2020} Let $A_0$, $\ldots$, $A_m$ be meromorphic functions in $\mathbb{C}^n$ such that there exists an integer $k \in \{0, \ldots, m\}$ satisfying 
\begin{align*}
\rho(A_k)>\max \{\rho(A_j): 0 \le j \le m, j \ne k\}\;\;\text{and}\;\;\delta_{A_k}(\infty)>0.
\end{align*}
If $f$ is a nontrivial meromorphic solution of linear partial difference equation 
\begin{align*}
A_m(z)f(z+c_m)+\ldots+A_1(z)f(z+c_1)+A_0f(z)=0,
\end{align*}
 where $c_1,\ldots,c_m$ are distinct values of $\mathbb{C}^n\setminus\{0\}$, then we have $\rho(f)\ge \rho(A_k)+1$.
\end{lem}

\begin{lem}\label{L2.3}\cite[Theorem 2.2]{Cao-Xu-2020} Let $f$ be a non-constant meromorphic function on $\mathbb{C}^n$ with 
\begin{align*}
\limsup\limits_{r\rightarrow \infty}\frac{\log T(r,f)}{r}=0,
\end{align*}
then $\parallel\;T(r,f(z+c))=T(r,f)+o(T(r,f))$ holds for any constant $c\in\mathbb{C}^n$.
\end{lem}

\begin{lem}\label{L2.4} \cite[Theorem 1.26]{Hu-Li-Yang-2003} Let $f:\mathbb{C}^n\to\mathbb{P}^1$ be a non-constant meromorphic function. Assume that 
$R(z, w)=\frac{A(z, w)}{B(z, w)}$. Then
\begin{align*} T\left(r, R_f\right)=\max \{p, q\} T(r, f)+O\Big(\sideset{}{_{j=0}^{p}}{\sum}T(r, a_j)+\sideset{}{_{j=0}^{q}}{\sum}T(r, b_j)\Big),\end{align*}
where $R_f(z)=R(z, f(z))$ and two coprime polynomials $A(z, w)$ and $B(z,w)$ are given
respectively $A(z,w)=\sum_{j=0}^p a_j(z)w^j$ and $B(z,w)=\sum_{j=0}^q b_j(z)w^j$.
\end{lem}

\begin{lem}\label{L2.5}\cite[Corollary 4.5]{Berenstein-Chang-Li-1994} Suppose $a_0 ,a_1,\ldots,a_m\; (m\geq 1)$ are meromorphic functions on $\mathbb{C}^n$ and $g_0,g_1,\ldots,g_m$ are entire functions on $\mathbb{C}^n$ such that $g_j-g_k$ are not constants for $0\leq j<k\leq m$. If the conditions 
\begin{align*}
\sideset{}{_{j=0}^m}{\sum} a_je^{g_j}=0
\end{align*}
and $T(r,a_i)=o(T(r,e^{g_j-g_k})),\;j\neq k$, $i=0,1,\ldots,m$
hold, then $a_i\equiv 0$ for all $i=0,1,\ldots,m$.
\end{lem}

\begin{lem}\label{L2.6}\cite{MS1} Let $f$ be a transcendental meromorphic function in $\mathbb{C}^n$ such that $\rho(f)<1$. Let $h>0$. Then there exists an $\varepsilon$-set $E$ in $\mathbb{C}^n$ such that
\begin{align*}
\frac{\partial_{z_i}(f(z+c))}{f(z+c)}\rightarrow 0
\end{align*}
for all $i\in\mathbb{Z}[1,n]$ and 
\begin{align*}
\frac{f(z+c)}{f(z)}\rightarrow 1
\end{align*}
as $||z||\rightarrow \infty$, where $z\in\mathbb{C}^n\backslash  E$ and $||c||\leq h$. Further, $E$ may be chosen so that for large $||z||$, where $z\not\in E$, the function $f(z)$ has no zeros or poles in $||\zeta-z||\leq h$.
\end{lem}

\begin{lem}\label{L2.7}\cite{MS1} Let $f$ be a transcendental meromorphic function in $\mathbb{C}^n$ of finite order $\rho$. 
Then there exists an $\varepsilon$-set $E$ in $\mathbb{C}^n$ such that
\begin{align*}
 \displaystyle \left|\frac{\partial_{z_i}(f(z))}{f(z)}\right|\leq ||z||^{\rho-1+\delta},
 \end{align*}
holds for all large values of $||z||$, where $z\in \mathbb{C}^n\backslash E$ and $\delta>0$ is a given constant. 
\end{lem}

\begin{lem}\label{L2.8} Let $P$ and $Q$ be two polynomials in $\mathbb{C}^n$ such that $P$ is non-constant. If $f$ is a finite order solution in $\mathbb{C}^n$ of the following equation
\begin{align*}
f(z+c)=e^{P(z)}\left(\frac{\partial f(z)}{\partial z_j}+Q(z)f(z)\right)
\end{align*}
where $c\in\mathbb{C}^n\backslash \{0\}$ and $j\in\{1,2,\ldots,n\}$, then $\rho(f)\geq \deg(P)+1$.
\end{lem}
\begin{proof} We have
\begin{align}\label{Eq3.2}
\frac{f(z+c)}{f(z)}=e^{P(z)}\left(\frac{\partial f(z)}{\partial z_j}/f(z)+Q(z)\right).
\end{align}
Applying Lemma 1.37 \cite{Hu-Li-Yang-2003}, we have
$m\left(\frac{\partial f}{\partial z_j}/f\right)=O(\log r)$
and by Lemma \ref{L2.1}, we obtain
\begin{align*}
m\left(r,f(z+c)/f(z)\right)=O\left(r^{\rho(f)-1+\varepsilon}\right)
 \end{align*}
holds for any $\varepsilon>0$. Since $m(r,Q)=O(\log r)$, from (\ref{Eq3.2}), we deduce that
\begin{align*}
m(r,e^P)=O(r^{\deg(P)})\leq O\left(r^{\rho(f)-1+\varepsilon}\right)+O(\log r),
\end{align*}
which shows that $\rho(f)\geq \deg(P)+1$.
\end{proof}

\subsection{{Proof of Theorem \ref{t2.2}}} By the given condition, we have
\begin{align}\label{t2.2:1.1}
\frac{\Delta_c^k f}{f}=e^{P},
\end{align}
where $P$ is an entire function in $\mathbb{C}^n$. Since $\rho(f)<2$, using Lemma 2.B to (\ref{t2.2:1.1}), we deduce that $P$ is a polynomial in $\mathbb{C}^n$ of degree at most one. If possible, suppose $\deg(P)=1$. Note that
\begin{align}\label{t2.2:1.2}
\Delta_c^{k} f(z)=\sideset{}{_{j=0}^{k}}{\sum}(-1)^{k-j}\binom{k}{j}f(z + jc).
\end{align}
Substituting (\ref{t2.2:1.2}) into (\ref{t2.2:1.1}), we get
\begin{align}\label{t2.2:1.3}
f(z + kc)+\sideset{}{_{j=1}^{k-1}}{\sum}(-1)^{k-j}\binom{k}{j}f(z + jc)+\bigl[(-1)^k - e^{P(z)}\bigr]f(z)= 0.
\end{align}
Applying Lemma \ref{L2.2} to (\ref{t2.2:1.3}), we deduce that $\rho(f)\ge \rho(e^P)+ 1=2$, which contradicts the fact that $\rho(f)<2$. Hence $P$ is a constant. If we take $d=e^P$, then from (\ref{t2.2:1.1}), we have 
\begin{align}\label{t2.2:1.4}
\Delta_c^{k} f \equiv df
\end{align}
for all $z\in\mathbb{C}^n$. If possible, suppose $\rho(f)<1$. Substituting (\ref{t2.2:1.2}) into (\ref{t2.2:1.4}) and then using Lemma \ref{L2.6}, we get
\begin{align*}
d=\frac{\Delta_c^{k} f(z)}{f(z)}=\sideset{}{_{j=0}^{k}}{\sum}(-1)^{k-j}\binom{k}{j}\frac{f(z + jc)}{f(z)}\to \sideset{}{_{j=0}^{k}}{\sum}(-1)^{k-j}\binom{k}{j}=(1-1)^k=0
\end{align*}
as $||z||\to\infty$ possibly outside an $\varepsilon$-set $E$ in $\mathbb{C}^n$, which is impossible. Hence $\rho(f)\geq 1$.

\subsection{{Proof of Theorem \ref{t2.3}}} By the given condition, we have
\begin{align}\label{t2.3:1.1}
\frac{f(z+c)-a(z)}{f(z)-a(z)}=e^{P(z)},
\end{align}
where $P$ is an entire function in $\mathbb{C}^n$. Applying Lemmas \ref{L2.3} and \ref{L2.4} to (\ref{t2.3:1.1}), we deduce that $T(r,e^P)=O(T(r,f))$ and so $\rho(e^P)\leq \rho(f)$. Since $\rho(f)$ is finite, Lemma 2.B yields $\deg(P)=\rho(e^P)\leq \rho(f)$. Note that $\infty$ is a Borel exceptional value of $f$. Since $\alpha\in\mathbb{C}$ is a Borel exceptional value of $f$, by Proposition \ref{Pro1}, we have 
\begin{align}\label{t2.3:1.2}
f(z)=H(z)e^{Q(z)}+\alpha,
\end{align}
where $H(z)$ is a non-zero entire function in $\mathbb{C}^n$ and $Q(z)$ is a polynomial in $\mathbb{C}^n$ such that $\rho(H)<\rho(f)=\rho(e^Q)=\deg(Q)$. Since $f$ has two distinct Borel exceptional values, by Proposition \ref{Pro3}, we have $\rho(f)=\mu(f)$. Then by Proposition 1.A, we have $T(r,H)=o(T(r,f))$. We now conclude from (\ref{t2.3:1.1}) and (\ref{t2.3:1.2}) that
\begin{align}\label{T2.3:1.3}
H(z+c)e^{Q(z+c)}-H(z)e^{P(z)+Q(z)}+(a(z)-\alpha)e^{P(z)}-(a(z)-\alpha)= 0.
\end{align}
\smallskip

\noindent
First, we suppose that $P$ is a non-constant polynomial. Clearly
$1 \le \deg(P) \le \rho(f) = \deg(Q).$
We now consider the following two cases.
\smallskip

\noindent
\textbf{Case 1.} Let $1 \le \deg(P) < \rho(f)=\deg(Q)$. Equation (\ref{T2.3:1.3}) can be rewritten as
\begin{align}\label{t2.3:1.3}
H(z+c)e^{Q(z+c)-Q(z)}-H(z)e^{P(z)}=(a(z)-\alpha)\bigl(1 - e^{P(z)}\bigr)e^{-Q(z)}.
\end{align}
Note that $\parallel T\left(r,H(z+c)e^{Q(z+c)-Q(z)}-He^{P}\right)=o(T(r,f))$ and $\parallel T\left(r,(a-\alpha)\bigl(1 - e^{P}\bigr)\right)=o(T(r,f))$. Since $\parallel T(r,H)=o(T(r,f))$, using Lemma \ref{L2.4} to \eqref{t2.3:1.2}, we obtain $\parallel T(r,f)=T(r,e^Q)+o(T(r,f))$. Finally from \eqref{t2.3:1.3}, we get $\parallel T(r,f)=o(T(r,f))$.
%Using Proposition 1.A, we see that the order of the left-hand side of (\ref{t2.3:1.3}) is less than $\rho(f)$, while the order of the right-hand side of (\ref{t2.3:1.3}) is equal to $\rho(f)$. 
This yields a contradiction.
\smallskip

\noindent
\textbf{Case 2.} Let $1 \le \deg(P)=\rho(f)=\deg(Q)=m$. Suppose
\begin{align*}
P(z)=\sideset{}{_{j=0}^m}{\sum} P_j(z)\quad \text{and}\quad Q(z)=\sideset{}{_{j=0}^m}{\sum} Q_j(z),
\end{align*}
where $P_j(z)$ and $Q_j(z)$ are homogeneous polynomials in $\mathbb{C}^n$ of degree $j\;(0\le j\le m)$ such that $P_m(z)$ and $Q_m(z)$ are non-constant.
Now we consider the following three sub-cases.
\smallskip

\noindent
\textbf{Sub-case 2.1.} Let $P_m-Q_m\equiv 0$. Equation (\ref{t2.3:1.3}) can be rewritten as
\begin{align}\label{t2.3:1.4}
H_{11}(z)e^{-Q(z)} + H_{12}(z)e^{P(z)} + H_{13}(z)e^{Q_{0}(z)}=0,
\end{align}
where $Q_{0}\equiv 0$ and
\begin{align*}
H_{11}= \alpha - a,\; H_{12}= -H,\;
H_{13}(z)= H(z+c)e^{Q(z+c)-Q(z)} + (a(z)-\alpha)e^{P(z)-Q(z)}.
\end{align*}
Note that
\begin{align*}
\deg(-Q-P)= \deg(-Q-Q_{0})= \deg(P-Q_{0})=m,
\end{align*}
and $T(r,a)=o(T(r,e^{Q})).$
Thus, for all $j = 1,2,3$, we have
\begin{align*}
\parallel T(r,H_{1j})= o\left(T\left(r,e^{-Q-P}\right)\right),\;\; \parallel T(r,H_{1j})= o\left(T\left(r,e^{-Q-Q_{0}}\right)\right)
\end{align*}
and $\parallel T(r,H_{1j})= o\left(T\left(r,e^{P-Q_{0}}\right)\right).$
Applying Lemma \ref{L2.5} to (\ref{t2.3:1.4}), we deduce that $H_{1j}\equiv 0$ for $j = 1,2,3$, which is a contradiction.
\smallskip

\noindent
\textbf{Sub-case 2.2.} Let $P_m+Q_m\equiv 0$. Equation (\ref{t2.3:1.3}) can be rewritten as
\begin{align*}
H_{21}(z)e^{Q(z+c)} + H_{22}(z)e^{P(z)} + H_{23}(z)e^{Q_{0}(z)} = 0,
\end{align*}
where $Q_{0}\equiv 0$ and
\begin{align*}
H_{21}(z) = H(z+c),\;\;H_{22}(z) = a(z) - \alpha\;\text{and}\; H_{23}(z) = -H(z)e^{Q(z)+P(z)} - (a(z)-\alpha).
\end{align*}
Proceeding in the same way as done in the proof Sub-case~2.1, we deduce that $H_{2j}\equiv 0$ for $j = 1,2,3$, which is a contradiction.
\smallskip

\noindent
\textbf{Sub-case 2.3.} Let $P_m(z)\pm Q_m(z)\not\equiv 0$. Equation (\ref{t2.3:1.3}) can be rewritten as
\begin{align*}
H_{31}(z)e^{Q(z+c)} + H_{32}(z)e^{Q(z)+P(z)}+ H_{33}(z)e^{P(z)} + H_{34}(z)e^{Q_{0}(z)}=0,
\end{align*}
where $Q_{0}\equiv 0$ and
\begin{align*}
H_{31}(z) = H(z+c),\;\;H_{32}(z)= -H(z),\;\;H_{33}(z)= a(z) - \alpha\;\text{and}\;H_{34}(z) = \alpha - a(z).
\end{align*}
Proceeding in the same way as done in the proof Sub-case~2.1, we deduce that $H_{3j}\equiv 0$ for $j=1,2,3,4$, which is a contradiction.
\vspace{1.2mm}

\noindent
Next we suppose that $P$ is a constant. Let $d=e^{P}$. If possible, suppose $d\neq 1$. Then equation (\ref{t2.3:1.3}) becomes
\begin{align}\label{t2.3:1.5}
H(z+c)e^{Q(z+c)-Q(z)}-dH(z)=(a(z)-\alpha)(1-d)e^{-Q(z)}.
\end{align}
Note that $\parallel T\left(r,H(z+c)e^{Q(z+c)-Q(z)}-dH\right)=o(T(r,f))$ and $\parallel T\left(r,(a-\alpha)\bigl(1 - d\bigr)\right)=o(T(r,f))$. Since $\parallel T(r,H)=o(T(r,f))$, using Lemma \ref{L2.4} to \eqref{t2.3:1.2}, we obtain $\parallel T(r,f)=T(r,e^Q)+o(T(r,f))$. Finally from \eqref{t2.3:1.5}, we get $\parallel T(r,f)=o(T(r,f))$, which is impossible.
%Note that
%\begin{align*}
%\rho\!\left((a(z)-\alpha)(1-d)e^{-H(z)}\right)=m= \rho(f)
%\end{align*}
%and
%\begin{align*}
%\rho\!\left(H(z+c)e^{Q(z+c)-Q(z)}-dH(z)\right)<m= \rho(f).
%\end{align*}
%So from (\ref{t2.3:1.5}), we get a contradiction. 
Hence $d=1$ and so (\ref{t2.3:1.1}) reduces to $f(z+c) \equiv f(z)$.

\subsection{{Proof of Theorem \ref{t2.4}}} According to Hartogs' theorem, a function $f$ defined on $\mathbb{C}^n$ is meromorphic on $\mathbb{C}^n$ if and only if, for each $j \in \{1, \dots, n\}$, $f$ is meromorphic in $z_j$ when the variables $z_1, \dots, z_{j-1}, z_{j+1}, \dots, z_n$ are fixed.
Given any fixed $n-1$ complex numbers $z_1, \dots, z_{j-1}, z_{j+1}, \dots, z_n$, consider the restriction 
$h(\xi) := f(z_1, \dots, z_{j-1}, \xi, z_{j+1}, \dots, z_n)$, which is a meromorphic function of $\xi$ in the complex plane. Then,
\begin{align*}
h'(\xi)= \frac{\partial f(z_1, \dots, z_{j-1}, \xi, z_{j+1}, \dots, z_n)}{\partial z_j}.
\end{align*}

\noindent
Thus, the equation \eqref{Eq3.1} becomes $h'(\xi)=h(\xi+c_j)$,
which demands that $h(\xi)$ is transcendental. Consequently $f(z_1, \dots, z_{j-1}, z_j, z_{j+1}, \dots, z_n)$ is transcendental in $z_j$, when the variables $z_1, \dots, z_{j-1}, z_{j+1}, \dots, z_n$ are fixed. Therefore it is easy to conclude that Eq. (\ref{Eq3.1}) has no rational solutions in $\mathbb{C}^n$. If possible, suppose $f$ is a transcendental meromorphic solution of the equation (\ref{Eq3.1}) such that $\rho(f)<1$. Then by Lemmas \ref{L2.6} and \ref{L2.7}, we have respectively
\begin{align}\label{t2.4:1.1}
\displaystyle \frac{f(z+c)}{f(z)}\rightarrow 1 \;\text{and}\; \left|\frac{\partial_{z_j}(f(z))}{f(z)}\right|=o(1),
\end{align}
for large values of $||z||$, where $z\in \mathbb{C}^n\backslash E$, $j\in\{1,2,\ldots,n\}$ and $E$ is an $\varepsilon$-set in $\mathbb{C}^n$. Using (\ref{t2.4:1.1}) to (\ref{Eq3.1}), we immediately get a contradiction. Hence $\rho(f)\geq 1$.

\subsection{{Proof of Theorem \ref{t2.5}}} Note that $0$ and $\infty$ are the Borel exceptional values of $f$.
By Proposition \ref{Pro1}, we have $f(z)=g(z)e^{Q(z)}$, where $g(z)$ is a non-zero entire function in $\mathbb{C}^n$ such that $g(0)\neq 0$ and $Q(z)$ is a polynomial in $\mathbb{C}^n$ such that 
\begin{align}\label{t2.5:1.0}
\rho(g)<\rho(f)=\rho(e^Q)=\deg(Q).
\end{align}
 Also by Proposition \ref{Pro3}, we have $\rho(f)=\mu(f)\in\mathbb{N}$. 
%Then by Proposition 1.A, we have $T(r,g)=o(T(r,f))$. 
Now from (\ref{Eq3.1}), we have
\begin{align}\label{t2.5:1.1}
\left(\frac{\partial g(z)}{\partial z_j}+g(z)\frac{\partial Q(z)}{\partial z_j}\right)e^{Q(z)-Q(z+c)}=g(z+c).
\end{align}
%where $j\in\{1,2,\ldots,n\}$.
\smallskip

\noindent
First we suppose that $Q(z)-Q(z+c)$ is non-constant. Clearly $\deg(Q)\geq 2$. Also from \eqref{t2.5:1.1}, we conclude that $g$ is a transcendental entire function. Take $\xi=(\xi_1,\ldots,\xi_n)\in\mathbb{C}^n\backslash \{0\}$. Define a holomorphic mapping $j_{\xi}:\mathbb{C}\to \mathbb{C}^n$ by $j_{\xi}(\zeta)=\zeta \xi$. Clearly
\begin{align*}
	f_{\xi}=f\circ j_{\xi}=(g\circ j_{\xi})e^{Q\circ j_{\xi}}=g_{\xi}e^{Q_{\xi}}:\mathbb{C}\to \mathbb{C}
\end{align*}
is a holomorphic mapping, where $g_{\xi}=g\circ j_{\xi}$ and $Q_{\xi}=Q\circ j_{\xi}$. Obviously $Q_{\xi}(\zeta)=Q(\xi_1 \zeta,\ldots,\xi_n \zeta)$, $Q_{\xi+c}(\zeta)=Q(\xi_1 \zeta+c_1,\ldots,\xi_n \zeta+c_n)$ and $g_{\xi+c}(\zeta)=g(\xi_1 \zeta+c_1,\ldots,\xi_n \zeta+c_n)$. In this case (\ref{t2.5:1.1}) reduces to 
\begin{align}\label{t2.5:1.1a}
	\left(\frac{d g_{\xi}(\zeta)}{d \zeta}+g_{\xi}(\zeta)\frac{d Q_{\xi}(\zeta)}{d \zeta}\right)e^{Q_{\xi}(\zeta)-Q_{\xi+c}(\zeta)}=g_{\xi+c}(\zeta).
\end{align}
 We can choose $\xi\in\mathbb{C}^n[0,1]\backslash \{0\}$ in such a way that $\deg(Q_{\xi})=\deg(Q)\geq 2$ and
$\deg(Q_{\xi}(\zeta)-Q_{\xi+c}(\zeta))=\deg(Q_{\xi}(\zeta))-1$. Also, we know that (see \cite[pp. 286]{Hu-Li-Yang-2003})
\begin{align}\label{t2.5:1.1b} 
	T_{g_{\xi}}(r,0)\leq \frac{1+\theta}{(1-\theta)^{2n-1}} T_{g}\left(\frac{r}{\theta},0\right),\quad \xi\in\mathbb{C}^n[0,1]\backslash \{0\}
\end{align}
where $0<\theta<1$. From inequality (\ref{t2.5:1.1b}), we deduce that
$\rho(g_{\xi})\leq \rho(g)$. Since $\deg(Q_{\xi})=\deg(Q)$, from \eqref{t2.5:1.0}, we obtain
\begin{align}\label{t2.5:1.1c}
\rho(g_{\xi}(\zeta))\leq \rho(g)<\deg(Q)=\deg(Q_{\xi}(\zeta)).
\end{align}
On the other hand, using lemma \ref{L2.8} to (\ref{t2.5:1.1a}), we deduce that 
\begin{align*}
\rho(g_{\xi}(\zeta))\geq \deg(Q_{\xi}(\zeta)-Q_{\xi+c}(\zeta))+1=\deg(Q_{\xi}(\zeta)),
\end{align*}
 which contradicts \eqref{t2.5:1.1c}.
\smallskip

\noindent
Next we suppose that $Q(z)-Q(z+c)$ is constant. 

\noindent
Let $g$ be a polynomial in $\mathbb{C}^n$. Then from \eqref{t2.5:1.1}, one can easily deduce that $\frac{\partial Q(z)}{\partial z_j}$ is a constant, say $a_j$. On integration, we get 
\[Q(z)=a_jz_j+\hat Q(z_1,\ldots,z_{j-1},z_{j+1},\ldots,z_n),\]
 where $\hat Q(z_1,\ldots,z_{j-1},z_{j+1},\ldots,z_n)$ is a polynomial in $\mathbb{C}^{n-1}$. Since $Q(z)-Q(z+c)$ is constant, it follows that $\hat Q(z_1+c_1,\ldots,z_{j-1}+c_{j-1},z_{j+1}+c_{j+1},\ldots,z_n+c_n)-\hat Q(z_1,\ldots,z_{j-1},z_{j+1},\ldots,z_n)$ is also a constant. Consequently, we may assume that
 \[Q(z)=a_1z_1+a_2z_2+\ldots+a_nz_n+\tilde Q(z),\]
  where $a_1,a_2,\ldots,a_n\in\mathbb{C}$ and $\tilde Q(z_1,\ldots,z_{j-1},z_{j+1},\ldots,z_n)$ is a polynomial in $\mathbb{C}^{n-1}$ such that $\tilde Q(z_1+c_1,\ldots,z_{j-1}+c_{j-1},z_{j+1}+c_{j+1},\ldots,z_n+c_n)\equiv \tilde Q(z_1,\ldots,z_{j-1},z_{j+1},\ldots,z_n)$.
 On the other hand from (\ref{t2.5:1.1}), we get
 \begin{align}\label{t2.5:1.2}
 	\frac{\partial g(z)}{\partial z_j}+a_jg(z)=e^{a_1c_1+a_2c_2+\ldots+a_nc_n}g(z+c).
 \end{align}
 Comparing the coefficients from the both sides of (\ref{t2.5:1.2}), we get $a_j=e^{a_1c_1+a_2c_2+\ldots+a_nc_n}$ and so from (\ref{t2.5:1.2}), we get
 \begin{align*}
 	\frac{\partial g(z)}{\partial z_j}=a_j(g(z+c)-g(z)).
 \end{align*}

 \noindent
 Let $g$ be a transcendental entire function in $\mathbb{C}^n$. Note that  $Q(z)-Q(z+c)$ is constant.
Therefore we may assume that $Q(z)=a_1z_1+a_2z_2+\ldots+a_nz_n+\tilde Q(z)$, where $a_1,a_2,\ldots,a_n\in\mathbb{C}$ and $\tilde Q(z)$ is a polynomial in $\mathbb{C}^n$ such that $ Q(z+c)\equiv \tilde Q(z)$. Also from (\ref{t2.5:1.1}), we get
\[\frac{\partial g(z)}{\partial z_j}+\left(a_j+\frac{\partial \tilde Q(z)}{\partial z_j}\right)g(z)=e^{a_1c_1+a_2c_2+\ldots+a_nc_n}g(z+c).\]

\vspace{0.1in}

 \noindent
{\bf Compliance of Ethical Standards:}\par

 \noindent
{\bf Conflict of Interest.} The authors declare that there is no conflict of interest regarding the publication of this paper.\par

 \noindent
{\bf Data availability statement.} Data sharing not applicable to this article as no data sets were generated or analysed during the current study.

\end{document}